\documentclass[12pt]{amsart}
\usepackage{times}
\usepackage{amssymb}
\usepackage{amsmath}
\usepackage{amsfonts}

\newcommand {\Holo}{\mathop{\rm Hol}\nolimits}

\def\1#1{\overline{#1}}
\def\2#1{\widetilde{#1}}
\def\3#1{\widehat{#1}}
\def\4#1{\mathbb{#1}}
\def\5#1{\frak{#1}}
\def\6#1{{\mathcal{#1}}}

\newcommand{\R}{\mathbb R}

\newcommand{\C}{\mathbb C}

\newcommand{\B}{\mathbb B}

\newcommand{\N}{\mathbb N}

\def\v{\varphi}
\def\Re{{\sf Re}\,}

\def\id{{\sf id}}

\def\diag{{\sf diag}}

\newtheorem{theorem}{Theorem}
\newtheorem{lemma}{Lemma}
\newtheorem{proposition}{Proposition}
\newtheorem{corollary}{Corollary}

\newtheorem{definition}{Definition}
\newtheorem{example}{Example}

\newtheorem{remark}{Remark}

\title[Normal forms for semigroups]{Normal forms and linearization of holomorphic dilation type semigroups in several variables}
\author[F. Bracci]{Filippo Bracci}
\address{Dipartimento Di Matematica\\
Universit\`{a} di Roma \textquotedblleft Tor Vergata\textquotedblright\ \\
Via Della Ricerca Scientifica 1, 00133 \\
Roma, Italy} \email{fbracci@mat.uniroma2.it}
\author[M. Elin]{Mark Elin}
\author[D. Shoikhet]{ David Shoikhet}
\address{Department of Mathematics\\
ORT Braude College \\ 21982 Karmiel,
Israel}\email{mark$\underline{\ }$elin@braude.ac.il}
\email{davs@braude.ac.il}

\thanks{This research is part of the European
Science Foundation Networking Programme HCAA}

\begin{document}

\maketitle

\begin{abstract}
In this paper we study commuting families of holomorphic
mappings in $\mathbb{C}^n$ which form abelian semigroups with
respect to their real parameter. Linearization models for
holomorphic mappings are been used in the spirit of
Schr\"oder's classical functional equation.
\end{abstract}

The one-dimensional linearization models for holomorphic
mappings and semigroups, based on Schr\"oder's and Abel's
functional equation have been studied by many mathematicians
for more than a century.

These models are powerful tools in investigations of asymptotic
behavior of semigroups, geometric properties of holomorphic
mappings and their applications to  Markov's stochastic
branching processes.

It turns out that solvability as well as constructions of the
solution of Schr\"oder's or Abel's functional equations
properly, depend on the location of the so-called Denjoy--Wolff
point of the given mappings or semigroups. In particular,
recently many efforts were directed to the study of semigroups
with a boundary Denjoy--Wolff point \cite{BE-PH, SD, AM-92,
RS-SD-96}.

Multidimensional cases are more delicate even when the
Denjoy--Wolff point is inside of the underlined domain. It appears
that the existence of the solution (the so-called K{\oe}nigs'
function) of a multidimensional Schr\"oder's equation depends also
on the resonant properties of the linear part of a given mapping
(or generator), and its relation to homogeneous polynomials of
higher degrees.

In parallel, the study of commuting mappings (or semigroups) is of
interest to many mathematicians and goes back to the classical
theory of linear operators, differential equations and evolution
problems.

In this paper we consider, in particular, the rigidity property
of two commuting semigroups. Namely, the question we study is
whether those semigroups coincide whenever the linear parts of
their generators at their common null point are the same.

\vspace{6mm}

Let $D$ be a domain in $\C^n$. We denote the set of holomorphic
mappings on $D$ which take values in a set $\Omega\subset\C^m$ by
$\Holo(D,\Omega)$. For each $f\in\Holo(D,\C^m)$, the Frech\'et
derivative of $f$ at a point $z\in D$ (which is understood as a
linear operator acting from $\C^n$ to $\C^m$ or $n\times
m$-matrix) will be denoted by $df_z$.

For brevity, we write $\Holo(D)$ for $\Holo(D,D)$. The set
$\Holo(D)$ is a semigroup with respect to composition operation.

\begin{definition} A family $\mathcal{S}=\{\v_t\}_{t\ge0}\subset\Holo(D)$ of
holomorphic self-mappings of $D$ is called a {\sl one-parameter
continuous semigroup} if the following conditions are
satisfied:

(i) $\v_{t+s}=\v_t\circ\v_s$ for all $s,\,t\ge0$;

(ii) $\lim\limits_{t\to0^+}\v_t(z)=z$ for all $z\in D$.
\end{definition}

It is more or less known that condition (ii) (the right continuity
of a semigroup at zero) actually implies its continuity (right and
left) on all of $\R^{+}=[0,\infty)$. Moreover, in this case the
semigroup is differentiable on $\R^{+}$ with respect to the
parameter $t\ge0$ (see \cite{BE-PH, SD, AM-92, RS-SD-96}). Thus,
for each $z\in D$ there exists the limit
\begin{equation}\label{gen}
\lim_{t\rightarrow 0^{+}}\frac{\v_t(z)-z}t=f(z),
\end{equation}
which belongs to $\Holo(D,\C^n)$. The mapping $f\in\Holo(D,\C^n)$
defined by (\ref{gen}) is called the {\it(infinitesimal)
generator} of $\mathcal{S}=\left\{\v_t\right\}_{t\ge0}$.

Furthermore, the semigroup $\mathcal{S}$ can be defined as a
(unique) solution of the Cauchy problem:
\begin{equation}\label{cauchy}
\left\{
\begin{array}{l}
{\displaystyle\frac{\partial \v_t(z)}{\partial
t}}=f(\v_t(z)),\quad t\ge0, \vspace{3mm}\\ \v_0(z)=z,\quad z\in D.
\end{array}
\right.
\end{equation}

\begin{definition}
We say that a semigroup $\left\{\v_t\right\}_{t\ge0}$ is  {\sl
linearizable} if there is a biholomorphic mapping
$h\in\Holo(D,\C^n)$ and a linear semigroup
$\left\{\psi_t\right\}_{t\ge0}$ such that
$\left\{\v_t\right\}_{t\ge0}$ conjugates with
$\left\{\psi_t\right\}_{t\ge0}$ by $h$, namely,
$h\circ\v_t=\psi_t\circ h$ for all $t\ge0$.
\end{definition}

Linearization methods for semigroups on the open unit disk in
$\C\left(=\C^1\right)$ have been studied by many mathematicians
(see, for example, \cite{Sis, SD-07, E-S-Y}). At the same time,
little is known about multi-dimensional cases. For example, in
\cite{Fab} and \cite{E-S-2004} the problem has been studied for
some special class of the so-called one-dimensional type
semigroups.

In this paper, we will concentrate on the case when a semigroup
has a (unique) interior attractive fixed point, i.e.,
$\lim\limits_{t\to\infty}\v_t(z)=\tau\in D\subset\C^n$ for all
$z\in D$. It is well known that this condition is equivalent to
that fact that the spectrum $\sigma(A)$ of the linear operator
(matrix) $A$ defined by $A:=df_\tau$ lies in the open left
half-plane (see \cite{AM-88} and \cite{RS-SD-96}) and
$d(\v_t)_\tau=e^{At}$. Usually, such semigroups are named {\it of
dilation type}. Thus, for the one-dimensional case, it is possible
to linearize the semigroup by solving Schr\"oder's functional
equation:
\[
h\left(\v_t(z)\right)=e^{f'(\tau)t}h(z)
\]
(see, for example, \cite{Sis, SD}).

\begin{remark} It should be noted that the latter equation involves the eigenvalue
problem for the linear semigroup $\left\{C_t\right\}_{t\ge0}$ of
composition operators on the space $\Holo(D,\C)$ defined by
$C_t:\,h\mapsto h\circ\v_t$.
\end{remark}

It is easy to show that the solvability of a higher dimensional
analog of Schr\"oder's functional equation
\begin{equation}\label{schroder}
h\left(\v_t(z)\right)=e^{At}h(z),\quad A=df_\tau,
\end{equation}
is equivalent to a generalized differential equation:
\begin{equation}\label{dif-schroder}
dh_z f(z)=A h(z).
\end{equation}

It seems that in general useful criteria (necessary and sufficient
conditions) for solvability of (\ref{dif-schroder}) are unknown.

Without loss of generality, let us assume that $\tau=0$.

\begin{proposition}\label{prop1}
Equation (\ref{schroder}), or equivalently, (\ref{dif-schroder})
is solvable if and only if there is a polynomial mapping
$Q:\,\C^n\mapsto\C^n$ with $Q(O)=O$ and $dQ_O=\id$, such that the
limit
\[
\lim_{t\to\infty}e^{-At}Q(\v_t(z))=:h(z),\quad z\in D,
\]
exists.
\end{proposition}

This proposition is based on the following notation and lemma.

By $\lambda(A)$ we denote the spectrum distortion index of the
matrix $A$, i.e.,
\[
\lambda(A):=\frac{\max\limits_{\alpha\in\sigma(A)}|\Re
\alpha|}{\min\limits_{\alpha\in\sigma(A)}|\Re \alpha|}\,.
\]
\begin{lemma}[see \cite{E-R-S-2004}]\label{lem1}
Let $g\in\Holo(D,\C^n)$ admit the expansion:
$g(z)=\sum\limits_{\ell\ge m}Q_\ell(z)$, where $Q_\ell$ is a
homogenous polynomial of order~$\ell$ and $m>\lambda(A)$. Then
\[
\lim_{t\to\infty}e^{-At}g(\v_t(z))=O,\quad\mbox{for all }\ z\in D.
\]
\end{lemma}

In many cases (and always --- in the one dimensional case), a
polynomial $Q$ in Proposition~\ref{prop1} can be chosen to be the
identity mapping, $Q(z)=z$ for all $z$. Moreover, in this case
$h\left(\v_t(z)\right)=e^{At}h(z)$, i.e., the mapping
$h(z)=\lim\limits_{t\to\infty}e^{At}\v(z)$ forms a conjugation of
a given semigroup $\left\{\v_t\right\}_{t\ge0}$ with the linear
semigroup $\left\{e^{At}\right\}_{t\ge0}$.

\begin{definition}
Let $\mathcal{S}=\{\v_t\}_{t\ge0}$ be a continuous one-parameter
semigroup of holomorphic self-mappings on a domain $D\subset\C^n$.
We say that $\mathcal{S}$ is  {\sl normally linearizable} if the
limit
\[
h(z)=\lim_{t\to\infty}e^{-At}\v_t(z),\quad z\in D,
\]
exists.
\end{definition}

A consequence of Lemma~\ref{lem1} is the following assertion.

\begin{proposition}\label{prop-normal1}
Let $\mathcal{S}=\{\v_t\}_{t\ge0}$ be a one-parameter semigroup
of\,\ holomorphic self-mappings on a domain $D\subset\C^n$
generated by $f\in\Holo(D,\C^n)$. If $f$ admits the expansion on
the series of homogenous polynomials:
$f(z)=Az+\sum\limits_{\ell\ge m}Q_\ell(z)$, where $Q_\ell$ is a
homogenous polynomial of order~$\ell$ and $m>\lambda(A)$, then the
semigroup $\mathcal{S}$ is normally linearizable.
\end{proposition}

In contrast with the one-dimensional case, for $n>1$ there are
semigroups which are not normally linearizable.

\begin{example}\label{ex2}
Let $\{\v_t\}_{t\ge0}$ be a semigroup in $\mathbb{C}^2$ defined by
\begin{eqnarray*}
\v_t(z_1,z_2)=\left( \begin{array}{c} z_1\exp\left(-(1+i)t\right) \vspace{2mm}\\
\Bigl[a{z_1}^2i\left(e^{-it}-1\right)+z_2\Bigr] e^{-(2+i)t}
\end{array}\right).
\end{eqnarray*}

It is easy to see that
\[
\lim_{t\to\infty}e^{-At}\v_t(z)=\lim_{t\to\infty}\left(\begin{array}{c}
z_1 \vspace{2mm}\\ a{z_1}^2i(\exp(-it)-1)+z_2 \end{array}\right)
\]
does not exist. Thus, this semigroup is not normally linearizable.

Just differentiating $\v_t$ at $t=0^+$ we find the semigroup
generator:
\begin{eqnarray*}
f(z_1,z_2)=\left( \begin{array}{c}  -(1+i)z_1 \vspace{2mm}\\
-(2+i)z_2+a{z_1}^2 \end{array}\right).
\end{eqnarray*}
For this generator we have $\lambda(A)=m=2$, i.e., $f$ does not
satisfy the conditions of Proposition~\ref{prop-normal1}.
\end{example}

\begin{proposition}\label{prop-normal2}
Let $D\subset \C^n$ be a domain containing $O$. Let
$\left\{\v_t\right\}$ be a continuous dilation semigroup which is
normally linearizable. If for some $t_0>0$ the semigroup element
$\v_{t_0}$ is a linear map, then all the elements $\v_t,\ t\ge0,$
are linear.
\end{proposition}

\begin{proof}
Denote $h(z):=\lim\limits_{t\to\infty}e^{-At}\v_t(z)$. Then for
all $s>0$ obviously
\[
h(\v_s(z)):=e^{As}\lim\limits_{t\to\infty}e^{-A(t+s)}\v_t(\v_s(z))=e^{As}h(z),
\]
i.e., $h$ is a linearizing conjugation for
$\left\{\v_t\right\}_{t\ge0}$. Since $\v_{t_0}=e^{At_0}$, we have
$\v_{t_0n}=e^{At_0n}$ and
\[
h(z):=\lim_{n\to\infty}e^{-At_0n}\v_{t_0n}(z)=z,
\]
so $h$ is the identity mapping. Therefore, $\v_s(z)=
h^{-1}\left(e^{As}h(z)\right)=e^{As}z$ for all $s\ge0$.
\end{proof}

Example~\ref{ex2} above shows that this fact is not generally
true. Indeed, for each $t_{\ell}=2\pi\ell,\ \ell\in\mathbb{Z},$
the semigroup element $\v_{t_{\ell}}$ is a linear mapping. Yet all
other elements $\v_t,\ t\not=2\pi\ell,$ are not linear.

\vspace{3mm}

An additional problem is that that with exception of the
one-dimensional case, linearizing conjugations may not be unique.

\begin{definition}
Let $\mathcal{F}=\left\{\v_s\right\}_{s\in\mathcal{A}}$ be a
family of holomorphic self-mappings of $D$. We say that
$\mathcal{F}$ is  {\sl uniquely linearizable} if there is a unique
mapping $h$ biholomorphic in $D$ and normalized by $h(O)=O,\
dh_O=\id$, such that
\[
h\circ\v_s =B_s\circ h ,\quad s\in\mathcal{A},
\]
where $\left\{B_s\right\}_{s\in\mathcal{A}}$ is an appropriate
family of linear operators on $\C^n$.
\end{definition}

\begin{remark}
Actually, it follows by the chain rule that $B_s=d(\v_s)_O$.
\end{remark}

\begin{remark}
A family $\mathcal{F}$ may consist of a single mapping
$F\in\Holo(D)$ as well as a discrete or continuous semigroup of
holomorphic self-mappings on $D$.
\end{remark}

Our next example shows that even linear diagonal mappings may not
be uniquely linearizable.

\begin{example}\label{ex3}
Consider a linear mapping $\psi=(\psi_1,\psi_2)$ with
\begin{eqnarray*}
\psi_1(z_1,z_2)= \frac{z_1}2, \qquad \psi_2(z_1,z_2)= \frac{z_2}4
\end{eqnarray*}
and a holomorphic normalized mapping defined by
\[
h(z_1,z_2)= \left(\begin{array}{c} z_1\vspace{2mm} \\
{z_1}^2+z_2 \end{array}\right).
\]
Then $h\circ\psi=\psi\circ h$, i.e., $h$ and also the identity
mapping $\id$ linearize $\psi$.
\end{example}

Actually, the question whether a linear mapping $\psi(z)=Bz$ is
uniquely linearizable can be formulated as the following rigidity
problem:

{\it When do the conditions
\[
Q\circ B=B\circ Q\quad\mbox{and} \quad Q'(O)=O
\]
on a holomorphic mapping $Q$ imply that $Q\equiv O\ $?}

\begin{remark}
In fact, it can be seen that if a matrix $B$ is diagonalazable and
$\sigma(B)=\{\beta_1,\ldots,\beta_n\}\subset\Delta$, then $\psi$
is uniquely linearizable if and only if
${\beta_1}^{k_1}\cdot{\beta_2}^{k_2}\cdot\ldots\cdot{\beta_n}^{k_n}\neq\beta_j$
for all $j=1,\ldots,n$ and $k\in\N^n$.
\end{remark}

\begin{theorem}\label{commut-th}
Let $D\subset \C^n$ be a domain containing $O$. Let
$\mathcal{S}=\left\{\v_t\right\}_{t\ge0}$ be a continuous
semigroup of dilation type, and let $\psi$ be a holomorphic
self-mapping of $D$ commuting with $\mathcal{S}$ such that
\begin{equation}\label{commut}
\psi\circ\v_t=\v_t\circ\psi
\end{equation}
for all $t\ge0$. If $\psi$ is uniquely linearizable by a
biholomorphic mapping $h:\,D\mapsto\C^n$, then all of the elements
of the semigroup $\mathcal{S}$ are linearizable by the same
mapping $h$.
\end{theorem}

\begin{proof} Let $B$ denote a linear operator on $\C^n$ defined by
$B=d\psi_O$. Also we denote $A=df_O$, where $f$ is the
infinitesimal generator of the semigroup $\mathcal{S}$. First, by
differentiating (\ref{commut}) at $O$ we obtain
$\left(d\psi_O\right)\circ e^{At}=e^{At}\circ d\psi_O$, i.e., $B$
commutes with the linear semigroup $\left\{e^{At}\right\}_{t\ge0}\
$ (in fact, $B$ commutes with $A$).

By our assumption, $h\circ\psi=B\circ h$. Therefore, for all
$t\ge0$ we have
\[
h\circ\psi\circ\v_t=B\circ h\circ\v_t.
\]
On the other hand, $h\circ\psi\circ\v_t= h\circ\v_t\circ\psi$ by
(\ref{commut}). Thus,
\[
e^{-At}\circ h\circ\v_t\circ\psi=e^{-At}\circ B\circ h\circ\v_t=
B\circ e^{-At}\circ h\circ\v_t.
\]

Denoting $h_1:=e^{-At}\circ h\circ\v_t$ one rewrites the latter
equality in the form
\[
h_1\circ\psi = B\circ h_1 .
\]
Since $h_1(O)=O,\ d(h_1)_O=\id$ and $\psi$ is uniquely
linearizable by $h$, we conclude that $h_1=e^{-At}\circ
h\circ\v_t=h$, or
\[
h\circ\v_t=e^{At}\circ h.
\]
The proof is complete.
\end{proof}

\begin{corollary}\label{cor4}
Let $D\subset \C^n$ be a domain containing $O$. Let
$\mathcal{S}=\left\{\v_t\right\}_{t\ge0}$ be a continuous
semigroup of dilation type. If there exists $t_0>0$ such that
$\v_{t_0}$ is uniquely linearizable by a biholomorphic mapping
$h:\,D\mapsto\C^n$, then all the elements of $\mathcal{S}$ are
linearizable by the same mapping $h$ which is a unique solution of
the differential equation (\ref{dif-schroder})
\[
dh_z f(z)=A h(z),
\]
normalized by the conditions $h(O)=O,\ dh_O=\id$.
\end{corollary}

\begin{corollary}
Let $D\subset \C^n$ be a domain containing $O$. Let
$\mathcal{S}_1=\left\{\v_t\right\}_{t\ge0}$ and
$\mathcal{S}_2=\left\{\psi_t\right\}_{t\ge0}$ be two continuous
semigroups on $D$ generated by mappings $f_1$ and $f_2$,
respectively. Suppose that $d(f_1)_O=d(f_2)_O=A$ with
$\Re\sigma(A)<0$ and that there exists $s_0>0$ such that

(i) $\psi_{s_0}$ is uniquely linearizable and

(ii) $\psi_{s_0}$ commutes with the semigroup $\mathcal{S}_1$ such
that $\psi_{s_0}\circ\v_t=\v_t\circ\psi_{s_0}$ for all $t\ge0$.

Then the semigroups coincide.
\end{corollary}

\begin{proof}
By our assumption, there is a unique biholomorphic mapping $h$
normalized by $h(O)=O,\ dh_O=\id$, such that
\[
h\circ\psi_{s_0} =e^{As_0}\circ h .
\]
Then Theorem \ref{commut-th} (or \,Corollary~\ref{cor4}) implies
that $h\circ\psi_s =e^{As}\circ h$ for all $s\ge0$. Since the
mapping $h$ is biholomorphic, we have:
\[
\psi_s=h^{-1}\circ\left(e^{As}\circ h\right).
\]

The commutativity of the mapping $\psi_{s_0}$ and the semigroup
$\mathcal{S}_1$ implies by the same Theorem~\ref{commut-th} that
all of the elements of $\mathcal{S}_1$ are linearizable by the
mapping $h$, that is, $h\circ\v_t =e^{At}\circ h$ for all $t\ge0$.
Thus
\[
\v_t=h^{-1}\circ\left(e^{At}\circ h\right).
\]
\end{proof}

\begin{remark}
If the semigroups $\mathcal{S}_1=\left\{\v_t\right\}_{t\ge0}$ and
$\mathcal{S}_2=\left\{\psi_t\right\}_{t\ge0}$ commute in the
sense: $\v_t\circ\psi_s=\psi_s\circ\v_t$ for all $t,s\ge0$, then
the conclusion that they coincide holds under a formally weaker
than condition (i) requirement that differential equation
(\ref{dif-schroder}) has a unique solution normalized by
${h(O)=O,}\ dh_O=\id$.
\end{remark}

\begin{corollary}
Let $D\subset \C^n$ be a domain containing $O$. Let
$\mathcal{S}_1=\left\{\v_t\right\}_{t\ge0}$ and
$\mathcal{S}_2=\left\{\psi_t\right\}_{t\ge0}$ be two commuting
semigroups on $D$ generated by mappings $f_1$ and $f_2$,
respectively. Suppose that $d(f_1)_O=d(f_2)_O=A$ with
$\Re\sigma(A)<0$. If $\lambda(A)<2$ then the semigroups coincide.
\end{corollary}

\vspace{3mm}

The use of the Poincar\'e--Dulac theorem (see, for example,
\cite{Arn}) is another approach to solve a linearization problem.

For simplicity, we assume in the sequel that $A$ is a diagonal
matrix, $A=\diag(\alpha_1\ldots,\alpha_n)$ with
$\Re\alpha_n\le\ldots\le\Re\alpha_1<0$.

Let $k:=(k_1,\ldots, k_n)\in \N^n$ be such that $|k|:=\sum k_j\ge
2$.

\begin{definition}\label{resonant}
We say that $A$ is {\sl resonant} (or the $n$-tuple
$(\alpha_1,\ldots\alpha_n)$ of the eigenvalues of $A$ is resonant)
if for some $\ell=1,\ldots, n$
\[
(\alpha,k):=\sum_{j=1}^n k_j\alpha_j= \alpha_\ell.
\]
Such a relation is called a {\sl resonance}. The number $|k|$ is
called the {\sl order} of the resonance.
\end{definition}

If $\alpha_\ell=(\alpha,k)$, we call any map $G:\C^n\mapsto\C^n$
{\sl resonant monomial} if it has the form $G(z)=(g_1(z),\ldots,
g_n(z))$ with $g_j\equiv 0$ for $j\neq \ell$ and $g_\ell(z)=a
z^k$.

\begin{lemma}\label{resonances}
If $\,\Re\alpha_n\le\ldots\le\Re\alpha_1<0$ then there is at most
a finite number of resonances for $\alpha$. Moreover, if
$\alpha_j=(k,\alpha)$ then  $k_j=\ldots=k_n=0$.
\end{lemma}

\begin{proof}
Both statements follow from the simple observation that if
$\alpha_j=(k,\alpha)$, then $\Re \alpha_j=(k,\Re \alpha)$, and by
the ordering of $\alpha_j$.
\end{proof}

For simplicity of notation, let
\[
M_j:=\left\{
\begin{array}{l}
0, \quad \mbox{if there is no } k \mbox{ with } \alpha_j=(k,\alpha),\vspace{2mm}\\
\max\{|k|: \alpha_j=(\alpha,k)\}\quad \mbox{otherwise}.
\end{array}\right.
\]
and $M(\alpha):=\max\{M_j:j=1,\ldots, n\}.$

A vector polynomial map  $R:\C^n\mapsto \C^n,\ R(O)=O,$ is {\sl
triangular} if by switching coordinates $R(z)=(R_1(z),\ldots,
R_n(z))$ assumes the form
\[
R_j(z)=a_j z_j + r_j(z_1,\ldots, z_{j-1}),\quad j=1,\ldots, n
\]
where $r_j$ is a polynomial.

\begin{theorem}
Let $D\subset \C^n$ be a domain containing $O$. Let
$\{\v_t\}_{t\ge0}$ be a continuous dilation type semigroup
generated by $f\in\Holo(D,\C^n)$ with $df_O=A$. Then there exists
an injective holomorphic map $h:\,D\mapsto \C^n$ (independent of
$\ t$) such that $h(O)=O$, $dh_O=\id$ and
\[
h \circ \v_t = P_t\circ h,
\]
where $P_t(z)=e^{At}z+R_t(z)$ is a triangular polynomial group of
automorphisms of $\C^n$ whose degree is less than or equal to
$M(\alpha)$, and $R_t(z)$ containing only resonant monomials. In
particular, if there are no resonances then $\{\v_t\}_{t\ge0}$ is
linearizable.
\end{theorem}

\begin{proof}
Let $\v_t(z)=e^{At}z+\sum\limits_{|m|\ge2} P_{m,t}(z)$ be the
homogeneous expansion at $O$ (which is defined on a small ball
containing $O$ and contained in $D$). It follows from the theory
of semigroups of holomorphic maps that each $P_{m,t}(z)$ is real
analytic in $t$.

By our assumption, $A$ is diagonal and the convex hull in $\C$ of
its eigenvalues does not contain $0$. Therefore by the classical
Poincar\'e--Dulac theorem, there exist an open neighborhood $U$ of
$O$ and a holomorphic map $h:\,U\mapsto \C^n$ normalized by
$h(O)=O$ and $dh_O=\id$ such that
$dh_{z}(f(z))=\widehat{f}(h(z))$, where $\widehat{f}(z)=Az+T(z)$
with $T$ being a polynomial vector field containing only resonant
monomials.

The semigroup $\{\v_t\}_{t\ge0}$ is (locally around $O$)
conjugated to the semigroup $\{\psi_t\}_{t\ge0},\ \psi_t=h \circ
\v_t\circ h^{-1}$, generated by $\widehat{f}=A+T$. Since $T$
contains only resonant monomials and $\Re \alpha_n\leq \ldots\leq
\Re \alpha_1<0$, Lemma~\ref{resonances} implies that $\widehat{f}$
is triangular, i.e., $\{\psi_t\}_{t\ge0}$ satisfies the following
system:
\[
\begin{cases}
\stackrel{\cdot}{x_1}=\alpha_1 x_1\\
\stackrel{\cdot}{x_2}=\alpha_2 x_2+r_2(x_1)\\\ldots\\
\stackrel{\cdot}{x_n}=\alpha_n x_n+r_n(x_1,x_2,\ldots, x_{n-1}),
\end{cases}
\]
where the $r_j$'s are polynomials in $x_1,\ldots, x_{j-1}$
containing only resonant monomials. Such a system can be
integrated directly by first solving
$\stackrel{\cdot}{x_1}=\alpha_1 x_1$, then substituting such
solution into $\stackrel{\cdot}{x_2}=\alpha_2 x_2+r_2(x_1)$, and
so on. In the end, $\psi_t$ is of the form
\[
\psi_t(z)=(e^{t\alpha_1}z_1, e^{t\alpha_2}(z_2+R_{2,t}(z_1)),
\ldots, e^{t\alpha_n}(z_n+R_{n,t}(z_1,z_2,\ldots, z_{n-1}))),
\]
with $R_{j,t}$ a polynomial in $z_1,\ldots, z_{j-1}$ of (at most)
degree $M_j$ containing with only resonant monomials. Moreover,
$R_{j,t}$ depends also polynomially on $t$. It can be shown by
induction. It is true for $j=1$, so assume it is true for $j-1$.
Then the $l$-th component of $(\psi_t)$ for $l=1,\ldots, j-1$ is
of the form $\psi_{t,l}(z)=
e^{t\alpha_l}(z_l+R_{l,t}(z_1,z_2,\ldots, z_{l-1}))$ with
$R_{l,t}$ a polynomial in $z_1,\ldots, z_{l-1}$ of degree at most
$M_l$ and depending polynomially on $t$. Substituting these into
the differential equation $\stackrel{\cdot}{x_j}=\alpha_j
x_j+r_j(x_1,x_2,\ldots, x_{j-1})$, one obtains
\begin{eqnarray*}
\stackrel{\cdot}{x_j}=\alpha_j
x_j+r_j(e^{\alpha_1t}z_1,e^{t\alpha_2}(z_2+R_{2,t}(z_1)),\ldots\\
\ldots, e^{t\alpha_{j-1}}(z_{j-1}+R_{j-1,t}(z_1,z_2,\ldots,
z_{j-2}))).
\end{eqnarray*}
Therefore the solution is of the form $e^{\alpha_j t} g(t)$ for
some function $g$ such that $g(0)=z_j$ and
\[
\stackrel{\cdot}{g}(t)=e^{-\alpha_j t}r_j(x_1,x_2,\ldots,
x_{j-1}).
\]
Now, $r_j$ contains only resonant monomials for $\alpha_j$. Let
$z^m$ be such a resonant monomial. Then, taking into account that
$m_j=\ldots=m_n=0$ by Lemma~\ref{resonances}, it follows
\[
 z^m=a z_1^{m_1}\cdots
 z_{j-1}^{m_{j-1}}=e^{(m,\alpha)t}[z_1^{m_1}\cdots (z_{j-1}+R_{j-1,t}(z_1,z_2,\ldots,
 z_{j-2}))^{m_{j-1}}].
\]
Hence
\[
\stackrel{\cdot}{g}(t)=e^{(-\alpha_j+(m,\alpha))
t}[z_1^{m_1}\cdots (z_{j-1}+R_{j-1,t}(z_1,z_2,\ldots,
 z_{j-2}))^{m_{j-1}}],
\]
and, being $\alpha_j=(m,\alpha)$, then actually
\[
\stackrel{\cdot}{g}(t)= z_1^{m_1}\cdots
(z_{j-1}+R_{j-1,t}(z_1,z_2,\ldots,
 z_{j-2}))^{m_{j-1}}.
\]
Since this holds for all resonant monomials in $r_j$, this
proves that $R_{t,j}(z)$ is a polynomial in both $z_1,\ldots,
z_{j-1}$ and $t$. The degree of $R_{t,j}$ is at most $M_j$
because it contains only resonant monomials for $\alpha_j$.
This proves the induction and the claim about the $R_{j,t}$'s.

This fact implies that $\psi_{-t}(z)$ is well defined for all
$t\geq 0$ and $z\in \C^n$. Therefore, $\{\psi_t\}_{t\in\R}$ is a
group of polynomial automorphisms of $\C^n$.

Finally, since $O$ is an attracting fixed point by hypothesis,
then $h$ can be extended to all $D$ by imposing
$h(w)=\psi_{-t}(h(\v_t(w)))$ for all $w\in D$.
\end{proof}

\begin{example}\label{ex1}
For $n=2$ there is only one possible resonance, namely,
$\alpha_2=m\alpha_1$. Hence, up to conjugation, the dilation
semigroups in $\C^2$ are of the form:
\[
\v_t(z)=(e^{\alpha_1 t}z_1, e^{\alpha_2 t}(z_2+atz_1^m))
\]
for some $a\in \C$.
\end{example}

So, if the matrix $A=df_\tau$ is resonant, it may happen that all
elements of the semigroup generated by $f$ are not linearizable.
In this connection the following question arises naturally.
Suppose that one of the elements of the semigroup
$\mathcal{S}=\{\v_t\}_{t\ge0}$ (say, $\v_{t_0}$) is linearizable.
Find conditions which ensure that all other elements $\v_t,\
t\not=t_0,$ are linearizable too.

To answer this question we need the following notion.

\begin{definition}\label{star}
We say that the matrix $A=\diag(\alpha_1\ldots\alpha_n)$ has {\sl
pure real resonance} if there are $j=1,\ldots,n$ and $k\in\N^n$
such that $\Re\alpha_j=\Re(\alpha,k)$ but
$\alpha_j\not=(\alpha,k)$.
\end{definition}

In particular, if all eigenvalues $\alpha_j$ have the same
argument, then $A$ has not pure real resonance.

\begin{theorem}\label{linear}
Let $D\subset \C^n$ be a domain containing $O$. Let
$\{\v_t\}_{t\ge0}$ be a continuous dilation semigroup generated by
$f\in\Holo(D,\C^n)$ with $df_O=A$, where $A$ has not pure real
resonance. If there exists $t_0>0$ such that $\v_{t_0}$ is
linearizable by biholomorphic mapping $h:D\mapsto \C^n,\ h(O)=O$.
Then the semigroup $\{\v_t\}_{t\ge0}$ is linearizable by $h$.
\end{theorem}

Not that even for the non-resonant case Theorem~3 completes
Theorem~2 since it asserts the following fact: if $h\in\Holo(D,
\C^n)$ is a linearizing mapping for $\v_{t_0}$, it also can serve
as a linearizing mapping for all $\v_t,\ t\ge0$.

\begin{proof}
Let us define $\psi_t:=h\circ \v_t \circ h^{-1}$. Then $\psi_t$ is
a semigroup on $h(D)$.

Let $\psi_t(z)=e^{At}z+\sum_m P_{m,t}(z)$ be the homogeneous
expansion at $O$ (which is defined on a small ball containing $O$
and contained in $g(D)$), where $m\geq 2$ is the least positive
integer such that $P_{m,t}\not\equiv 0$ for all $t,z$. If the
theorem holds then $m=+\infty$ (namely, $(\psi_t)$ is linear).
Seeking a contradiction, we assume that $m<+\infty$.

It follows from the theory of semigroups of holomorphic maps that
each $P_{m,t}(z)$ is real analytic in $t$.

Since by hypothesis $\psi_{t_0}=h\circ \v_{t_0} \circ h^{-1}$
is linear, then $P_{m,t_0}\equiv 0$.

Now, from $\psi_{t+s}=\psi_t\circ \psi_s$ it follows that
\begin{equation}\label{uno}
P_{m,t+s}(z)=e^{At}P_{m,s}(z)+P_{m,t}(e^{As}z).
\end{equation}
Write $P_{m,t}(z)=(\sum_{|k|=m}p^1_k(t)z^k,\ldots,
\sum_{|k|=m}p^n_k(t)z^k)$, where, as usual,
$z^k=z_1^{k_1}\cdots z_n^{k_n}$. From \eqref{uno} it follows
that for $j=1,\ldots, n$
\[
p^j_{k}(t+s)=e^{\alpha_j t}p_k^j(s)+p_k^j(t) e^{(\alpha,k)s}.
\]
Differentiating such an expression with respect to $t$ and setting
$t=0$, we obtain the following differential equation:
\begin{equation}\label{diff1}
\frac{d}{dt}p^j_{k}(s)=\alpha_jp_k^j(s)+a_k^je^{(\alpha,k)s},
\end{equation}
where we set $a_k^j=\frac{dp_k^j(t)}{dt}|_{t=0}$. There are two
cases:

\noindent (1) if $\Re\alpha_j\not=\Re(\alpha,k)$,  then imposing
the condition $p_k^j(0)=0$, equation~\eqref{diff1} has the
solution
\begin{equation}\label{due}
p^j_k(t)=a^j_k \frac{e^{(\alpha,k)t}-e^{\alpha_j
t}}{(\alpha,h)-\alpha_j}.
\end{equation}

\noindent (2) if $\Re\alpha_j=\Re(\alpha,k)$, then by our
assumption $\alpha_j=(\alpha,k)$. In this case, imposing the
condition $p_k^j(0)=0$, equation~\eqref{diff1} has the solution
\begin{equation}\label{due-res}
p^j_k(t)=a^j_k e^{t\alpha_j}t.
\end{equation}

By \eqref{due} and \eqref{due-res} it follows that
$p_k^j(t_0)=0$ if and only if $p_k^j(t)=0$ for all $t\geq 0$,
and hence $P_{m,t_0}\equiv 0$ if and only if $P_{m,t}\equiv 0$
for all $t\geq 0$, reaching a contradiction with our
hypothesis.
\end{proof}

Example~\ref{ex2} above shows that if $A$ has pure real resonance,
Theorem~\ref{linear} fails.

\begin{corollary}\label{cor1}
Let $\mathcal{S}=\{\v_t\}_{t\ge0}$ be a continuous semigroup of
dilation type generated by $f\in\Holo(D,\C^n)$ with
$df_O=A=\diag(\alpha_1,\ldots,\alpha_n)$. Suppose that there is
$t_0>0$ such that $\v_{t_0}$ is a linear mapping. Assume that one
of the following conditions holds:

(i) $A$ has not pure real resonance;

(ii) $e^{(\alpha,k)t_0}\not=e^{\alpha_jt_0}$ for all
$j=1,\ldots,n$ and $k\in\N^n$.

Then all elements of $\mathcal{S}$ are linear mappings.
\end{corollary}

\begin{proof}
If condition (i) holds, the assertion follows immediately by
Theorem~\ref{linear}.

Assume that condition (ii) holds. First, we show that $\v_{t_0}$
is uniquely linearizable. Indeed, let $h(z)=z+\ldots$ be a
linearizing mapping different from $\id$. This means that
$h\circ\v_{t_0}=\v_{t_0}\circ h$ and for some $j=1,\ldots,n,\,$
the $j$-th coordinate of $h$ contains a non-zero monomial
$a_k{z_1}^{k_1}\ldots{z_n}^{k_n}$ with $|k|\ge2.$ Therefore,
\[
h_j\left(e^{\alpha_1t_0}z_1,\ldots,e^{\alpha_nt_0}\right)
=e^{\alpha_jt_0}h_j(z),
\]
and so
\[
a_ke^{(\alpha,k)t_0}z^k=a_ke^{\alpha_jt_0}z^k.
\]
The contradiction provides that $\v_{t_0}$ is uniquely
linearizable by the identity mapping $\id$.

Now, Corollary~\ref{cor4} implies that the all mappings $\v_t,\
t\ge0,$ are linearizable by the identity mapping. Hence, they are
linear.
\end{proof}

Combining Corollary~\ref{cor1} with
Proposition~\ref{prop-normal2}, we get the following result.

\begin{corollary}
Let $\B^n$ be the unit ball of \,$\C^n$ and let
$\mathcal{S}=\{\v_t\}_{t\ge0}$ be a continuous semigroup of
dilation type generated by $f\in\Holo(\B,\C^n)$ with
$df_O=A=\diag(\alpha_1,\ldots,\alpha_n)$. Suppose that there is
$t_0>0$ such that $\v_{t_0}$ is a linear fractional self-mapping
of $\B^n$. Assume that one of the following conditions holds:

(i) $A$ has not pure real resonance;

(ii) $\mathcal{S}$ is normally linearizable;

(iii) $e^{(\alpha,k)t_0}\not=e^{\alpha_jt_0}$ for all
$j=1,\ldots,n$ and $k\in\N^n$.

Then for all $t\geq 0$ the mapping $\v_{t}$ is a linear fractional
self-map of\, $\B^n$.
\end{corollary}

\begin{proof}
According to \cite[Thm. 3.2 and Rmk. 3.4]{B-C-DM} (and its proof)
there exists $h:\,\B^n\mapsto\C^n$ a linear fractional mapping
fixing $O$  such that $h\circ \v_{t_0} \circ h^{-1}$ is linear. By
Corollary~\ref{cor1} and Proposition~\ref{prop-normal2}, it
follows that $h\circ \v_{t} \circ h^{-1}$ is linear for all
$t\ge0$. Therefore, $\v_t$ is the composition of linear fractional
maps and hence linear fractional for all $t\ge0$.
\end{proof}

\begin{corollary}
Let $\mathcal{S}=\{\v_t\}_{t\ge0}$ be a continuous semigroup of
dilation type generated by $f\in\Holo(\B,\C^n),\
f(z)=Az+\sum\limits_{\ell\ge m}Q_\ell(z)$, where $Q_\ell$ is a
homogenous polynomial of order~$\ell$ and $m>\lambda(A)$. If for
some $t_0>0$, the semigroup element $\v_{t_0}$ is a linear
(respectively, linear fractional) mapping, then all the elements
of $\mathcal{S}$ are linear (respectively, linear fractional)
mappings.
\end{corollary}

A direct consequence of our Theorems~2 and 3 and a recent Forelli
type extension theorem (see \cite[Theorem 6.2]{K-P-S}) is the
following assertion.

\begin{corollary}
Let $\mathcal{S}=\{\v_t\}_{t\ge0}$ be a continuous semigroup of
dilation type generated by $f\in\Holo(D,\C^n)$ with
$df_O=A=\diag(\alpha_1,\ldots,\alpha_n)$, where all eigenvalues
$\alpha_j$ have the same argument. Suppose that a function $F$
defined on $D$ is real analytic at $O$, and that its restrictions
to the integral curves of the vector field $f$ are holomorphic. If
at least one of the following conditions holds:

(i) $A$ is not resonant,

\noindent or

(ii) there is $t_0$ such that $\v_{t_0}$ is linearizable,

\noindent then $F$ is holomorphic on $D$.
\end{corollary}

\end{document}